\documentclass{article}%
\usepackage{amsfonts}
\usepackage{amsthm}
\usepackage{amsmath}
\usepackage{amssymb}
\usepackage{graphicx}%
\usepackage{epsfig}
\usepackage{color}
\providecommand{\U}[1]{\protect\rule{.1in}{.1in}}

\def\subneq{\mathop{\raise 0.7ex \hbox{$\subset$}}\!\!\!\!\!\!{\raise -0.6ex\hbox{$\neq$}}\,}

\def\div{\mathop{\rm div}\nolimits}

\def\u{\underline}

\def\QQ{{\ \rlap {\raise 0.4ex \hbox{$\scriptscriptstyle |$}}\hskip -0.2em Q}}
\def\RR{{ I\!\!R}}

\def\1{{1\hskip-0.25em{\rm l}}}

\def\CC{{\ \rlap{\raise 0.4ex \hbox{$\scriptscriptstyle |$}}\hskip -0.2em C}}

\def\sobre#1#2{\lower 1ex \hbox{ $#1 \atop #2 $ } }
\def\bajo#1#2{\raise 1ex \hbox{ $#1 \atop #2 $ } }

\def\div{{\rm div}}
\def\ep{\varepsilon}

\def\p{\partial}

\def\O{\Omega}
\def\Oe{\Omega^\ep}

\def\u2{{u^\ep \over \ep^2 }}
\def\u3{{\displaystyle {\bar u}^\ep \over \ep^2 }}

\def\div{\mathop{\rm div}\nolimits}

\def\u{\underline}

\def\QQ{{\ \rlap {\raise 0.4ex \hbox{$\scriptscriptstyle |$}}\hskip -0.2em Q}}
\def\RR{{ I\!\!R}}

\def\1{{1\hskip-0.25em{\rm l}}}

\def\CC{{\ \rlap{\raise 0.4ex \hbox{$\scriptscriptstyle |$}}\hskip -0.2em C}}

\def\B{{\cal B}}

\def\sobre#1#2{\lower 1ex \hbox{ $#1 \atop #2 $ } }
\def\bajo#1#2{\raise 1ex \hbox{ $#1 \atop #2 $ } }

\def\div{{\rm div}}
  
\def\ve{\mathbf{v}^\ep}
\def\ep{\varepsilon}

\def\p{\partial}

\def\O{\Omega}
\def\Oe{\Omega^\ep}

\def\u2{{u^\ep \over \ep^2 }}
\def\u3{{\displaystyle {\bar u}^\ep \over \ep^2 }}
\def\p{\partial}
\newtheorem{theorem}{Theorem}

\newtheorem{corollary}[theorem]{Corollary}

\newtheorem{definition}[theorem]{Definition}

\newtheorem{lemma}[theorem]{Lemma}

\newtheorem{proposition}[theorem]{Proposition}
\newtheorem{remark}[theorem]{Remark}

\begin{document}

\title{Effective pressure interface law for transport phenomena between an unconfined fluid and a porous medium using homogenization}
\author{Anna Marciniak-Czochra \thanks{AM-C  was supported by ERC Starting Grant "Biostruct" and Emmy Noether Programme of German Research Council (DFG).}
\\IWR and BIOQUANT,
Universit\"at Heidelberg \\
Im Neuenheimer Feld 267,
69120 Heidelberg  \\ GERMANY \\ ({\tt anna.marciniak@iwr.uni-heidelberg.de})
\and Andro Mikeli\'c\thanks{The research of AM
was partially supported by  {the {\it Joint-Programme
Programme Inter Carnot Fraunhofer PICF} { FPSI-Filt: 
Modeling of fluid interaction with deformable
porous media with application to simulation of
processes in industrial filters}. He is grateful to the Ruprecht-Karls-Universit\"at Heidelberg and the Heidelberg Graduate School of Mathematical and Computational Methods for the Science (HGS MathComp) for giving him good working conditions through the W. Romberg Guest Professorship 2011-2013.}
 }\\Universit\'e de Lyon, Lyon, F-69003, France;\\Universit\'e
Lyon 1, Institut Camille Jordan, UMR 5208\\B\^at. Braconnier,  43, Bd du 11
novembre 1918,\\ 69622 Villeurbanne Cedex,
FRANCE\\(\texttt{mikelic@univ-lyon1.fr})  } \maketitle

\begin{abstract}   We
present modeling of the incompressible viscous flows in the domain containing unconfined fluid and a porous medium in the case when the flow in the unconfined domain dominates.  For such setting a rigorous derivation of the Beavers-Joseph-Saffman
interface condition was undertaken  by J\"ager and Mikeli\'c  [SIAM J. Appl. Math. \rm 60 (2000), p.1111-1127] using the homogenization method. So far the interface law for the pressure was conceived and confirmed only numerically. In this article we derive the Beavers and Joseph law for a general body force by estimating the pressure field  approximation. Different than in the Poiseuille flow case, the velocity approximation is not divergence-free and the precise pressure estimation is essential.
This new estimate allows us to justify rigorously the pressure jump condition using the Navier boundary layer, already used to calculate the constant in the law by Beavers and Joseph. Finally,  our results confirm that the position of the interface influences the solution only at the order of physical permeability and therefore the choice of this position does not pose problems.  \end{abstract}

\section{Introduction}

Slow viscous and incompressible simultaneous flow through an unconfined region and a porous medium  occurs in a wide range of industrial processes and natural phenomena. One of the classical problems is finding effective boundary conditions at a naturally
permeable wall, i.e., at the surface which separates a channel flow and a porous medium.

The effective laminar incompressible and viscous flow through a porous medium can be  described
using the Darcy's law.   The unconfined fluid flow in the channel is  governed by the Stokes system, or by the Navier-Stokes system if the inertia effects in the free fluid are
important. To model the coupling of both processes, it is necessary to put together two
systems of partial differential equations:  the second order system for the velocity
and the first order equation for the pressure,
\begin{gather}
-\mu \Delta \mathbf{u} +\nabla p = f \label{Stokes1}
 \\ \mbox{div } \mathbf{u} =0,
\label{Stokes2}
\end{gather}
in  the unconfined fluid region,
and the scalar second order equation for the pressure and the
first order system for the seepage velocity,
\begin{gather}
-\mu  \mathbf{v}^F  = K ( f- \nabla p^F) \label{Darcy1} \\ \mbox{div }  \mathbf{v}^F =0,
\label{Darcy2}
\end{gather}
in the porous medium.\vskip0pt
The orders of the corresponding
differential operators are different  and it is not clear what
conditions it is necessary to impose at the interface between the
free fluid and the porous part of the domain. One coupling
condition is based on the continuity of the
normal mass flux. However, it is not enough for determination of the
effective flow and it is necessary to specify more conditions.

Several laws of fluid dynamics in porous media were derived using homogenization. The most notable example is the Darcy's law, being the effective equation for one phase flow through a rigid porous medium. Its formal derivation using the 2-scale expansion goes back to the classical paper  by Ene and Sanchez-Palencia \cite{ESP}. This derivation was made mathematically rigorous by Tartar in reference \cite{Ta1980}. For the detailed proof in the case of a periodic porous medium we refer to the review papers  by Allaire \cite{All97},  and  by Mikeli\'c \cite{MIK00} and for a random statistically homogeneous porous medium to the paper of  Beliaev and Kozlov \cite{BeKo95}.

As in the derivation of Darcy's law, we would like to apply the homogenization technique to find the effective interface laws.  However, the assumption of statistical homogeneity of the domain,
which is necessary for the homogenization approach, is not valid close to the interface. Consequently,  deviations from the Darcy's law are expected in the thin layers near the interfaces. Furthermore, presence of such interfaces can significantly change the
structure of the model coefficients and lead to different  effective constitutive laws for the flow.

 It was  experimentally found by Beavers and  Joseph in \cite{BJ} that the jump of the
tangential component of the effective velocity at the interface is
proportional to the shear stress originating from the free fluid. This law was
justified at a physics level of rigor by  Saffman in
\cite{SAF}, where it was observed that the seepage velocity
contribution could be neglected leading to the law in the form
\begin{equation}\label{AUX}
    \sqrt{k} \frac{\partial v_{\tau}}{ \partial \nu} =\alpha  v_\tau +
O(k),
\end{equation}
where $\alpha$ is a dimensionless parameter
depending on the geometrical structure of the porous medium, $\ep$
is the characteristic pore size, and $ k$ is the
scalar permeability. $\nu$ denotes the unit normal vector at the
interface and $v_\tau$ is the slip velocity of the free fluid in
the channel.
    Saffman's
modification of the law by Beavers and Joseph has been widely accepted.

As an alternative to (\ref{AUX}), the continuity of the effective pressure was suggested by Ene and  Sanchez-Palencia in \cite{ESP}. While this interface law is acceptable from modeling point of view,  it should be noted that the well-posedness of the averaged
problem is not clear.

The law (\ref{AUX}) was rigorously justified by J\"ager and Mikeli\'c in \cite{JM00}.  Numerical calculations of the boundary layers for the experimental conditions of Beavers and Joseph are presented in \cite{JMN01}. They indicate appearance of a {\it pressure jump} at the interface. These issues were heuristically discussed in \cite{JaegMik09}.

In the experiment by Beavers and Joseph only the flows tangential to a naturally
permeable wall (a porous bed) were considered. In general, the situation is much more complicated and  many types
of interfacial conditions have been proposed, such as continuous tangential velocity
with discontinuous tangential shear stress introduced in  \cite{OTW1:95}
by Ochoa-Tapia and Whitaker, or
continuous tangential velocity and tangential shear stress in reference \cite{NeNa:74} by Neale and
Nader, or discontinuous tangential velocity and tangential shear
stress from \cite{CiKu:99} by Cieszko and  Kubik.
In particular, in  \cite{OTW1:95} the
continuity of the velocity and  the continuity of the "modified" normal stress were obtained at the interface using volume averaging. In order to perform the averaging it was necessary to assume  the Brinkman's flow in the porous part and a transition layer between  the two domains.   Numerical study of the hydrodynamic boundary condition at
the interface between a porous and plain medium was performed by Sahraoui and
Kaviany \cite{SahKav92}.    Numerical implementation of the effective interface couplings was presented in \cite{Quat} and in \cite{IL}.
 Nevertheless, determination of the practical and relevant first-order
interface conditions between the pure fluid and the porous matrix remains an
open question that could be treated using the technique  developed in reference \cite{JaMi2}.

This paper is a continuation of  works \cite{JM00} and \cite{JMN01} and constitutes a step forward in the development of the rigorous approach to model effective interface laws for the transport phenomena between an unconfined fluid and a porous medium.  We  depart beyond justification of the law (\ref{AUX}) developed in \cite{JM00} and {\it undertake a rigorous derivation}   of the interface laws for  the viscous flow in a long channel in contact with a porous bed. The macroscopic model  derived links pressure jump with the shear stress of the unconfined fluid at the interface, an effect which was predicted based on numerical simulations in reference \cite{JMN01}.  Derivation of  the law of Beavers  and Joseph is based on the procedures proposed in \cite{JM00} and discussed in \cite{JaegMik09}; however it is nontrivially adjusted to the new setting involving a general body force.  We consider a  general situation when the flow in the unconfined region dominates. Nevertheless, even if the flow is much less important in the porous part, the pressures are of the same order of magnitude. Hence finding and justifying the interface law for the pressure is of fundamental interest.

The review paper \cite{JaegMik09} was concluded with the sentence
"Proving the error estimate for the pressure approximation in the porous bed $\Omega^{\varepsilon}_2$ remains an open problem".  We solve this problem and present a mathematically rigorous derivation of the pressure jump interface law, which is the next order correction of the Beavers-Joseph law.
  We
obtain the effective equations heuristically and then rigorously justify
them. Combination of  homogenization and boundary layer approaches is used to achieve this end.  Study of such complex flows leads to artificial compressibility effects in the upscaling process. In this paper we develop the required estimate of the pressure. Our main results are the following:
\begin{enumerate}
 \item Confirmation of Saffman's form of the law by Beavers and Joseph in the more general setting \begin{equation}\label{BJ1}
u^{eff}_1 = -\ep C^{bl}_1 \frac{\p u^{eff}_{1} }{ \p x_2}+O(\ep^2) ,    \end{equation}
 where $u^{eff}$ is the average
over the characteristic pore opening at the naturally permeable wall. Physical permeability is given by $k=k^\ep = \ep^2 K$ and the
 constant in (\ref{BJ1}) is proportional to
$\sqrt{k^\ep}$. The error is of order $k^\ep$, as remarked by
Saffman in \cite{SAF}. It is important to point out that the
parameter ${\alpha}$ from  expression (\ref{AUX}) is determined
taking into account the auxiliary problems, which we formulate later in (\ref{BJ4.2})-(\ref{4.6}) and
(\ref{4.15}), and that it is given by  $\displaystyle { \alpha} = - \frac{1}{ \ep
C^{bl}_1}>0 $.
  \item Interface between the unconfined flow and the porous bed is an artificial mathematical boundary
and it can be chosen in a layer having the pore size
thickness. We show that a perturbation of the interface position of the order
$O(\ep)$ implies a perturbation in the solution of $O(\ep^2)$.
Consequently, it influences the result only at the next order of the asymptotic
expansion.
  \item We obtain a uniform
bound on the pressure approximation. Furthermore, we prove that there is a jump of the effective pressure on the interface and that it is proportional to the free fluid flow shear at the boundary. The proportionality constant is calculated from the boundary layer problem (\ref{BJ4.2})-(\ref{4.6}). Homogenization leads to the discontinuity of the effective pressure field at
the interface, which differs from the pressure interface
continuity law proposed in reference \cite{ESP}. If the boundary layer pressure is neglected, the pressure in the neighborhood of the interface is poorly approximated.
\end{enumerate}
{Here, we remark that some classes of problems, like infiltration into the porous medium, are characterized by the velocity field of the same order in both domains. Such situation requests much larger body force in the porous part than in  unconfined. Some situations of this kind were considered in \cite{JaMi2}. In this paper, the body force is of order $O(1)$ in both domains.}

The paper is organized as follows. In Section \ref{StatPb} we formulate the problem and main results.  Section \ref{BJSec} is devoted to the proof of the results. We conclude the paper with two short appendices recalling the notion of very weak solutions and definition and properties of the Navier boundary layer.

\section{Statement of the problem and of the results}\label{StatPb}
\subsection{Definition of the geometry} \label{ExpDarcy}

Let $L, h$ and $H$ be positive real numbers. We consider a two dimensional periodic porous medium $\Omega_2 =
(0,L)\times (-H, 0)$ with a periodic arrangement of the pores. The formal
description goes along the following lines: \vskip0pt
First, we define the geometrical
structure inside the unit cell $Y = (0,1)^2$. Let
$Y_s$ (the solid part) be a closed strictly included subset of $\bar{Y}$, and $Y_F =
Y\backslash Y_s$ (the fluid part). Now we make a periodic
repetition of $Y_s$ all over $\mathbb{R}^2$ and set $Y^k_s = Y_s + k $, $k
\in \mathbb{Z}^2$. Obviously, the resulting set $E_s = \bigcup_{k \in
\mathbb{Z}^2} Y^k_s$ is a closed subset of $ \mathbb{R}^2$ and $E_F =  \mathbb{R}^2
\backslash E_s$ in an open set in $ \mathbb{R}^2$.  We suppose that
 $Y_s$ has a boundary of class $ C^{0,1}$, which is locally located on
one side of their boundary. Obviously,  $ E_F $ is connected and
$E_s$ is not.  \vskip3pt Now we notice that $\Omega_2$ is covered with a
regular mesh of size $ \varepsilon$, each cell being a cube
$Y^{\varepsilon}_i$, with $1 \leq i \leq N(\varepsilon) = \vert
\Omega_2 \vert  \varepsilon^{-2} [1+ o(1)]$. Each cube
$Y^{\varepsilon}_i$ is homeomorphic to $Y$, by linear homeomorphism
$\Pi^{\varepsilon}_i$, being composed of translation and a homothety
of ratio $1/ \varepsilon$.

We define
$\displaystyle
Y^{\varepsilon}_{S_i} = (\Pi^{\varepsilon}_i)^{-1}(Y_s)\qquad
\hbox{ and }\quad Y^{\varepsilon}_{F_i} =
(\Pi^{\varepsilon}_i)^{-1}(Y_F).
$
For sufficiently small $\varepsilon > 0 $ we consider the set
$ \displaystyle T_{\varepsilon} = \{k \in  \mathbb{Z}^2  \vert  Y^{\varepsilon}_{S_k} \subset
\Omega_2 \} $
and define
$$
O_{\varepsilon} = \bigcup_{k \in T_{\varepsilon}}
Y^{\varepsilon}_{S_k} , \quad S^{\varepsilon} = \partial
O_{\varepsilon}, \quad \Omega^{\varepsilon}_2 = \Omega_2 \backslash
O_{\varepsilon} =\O_2 \cap \ep E_F$$ Obviously, $\partial
\Omega^{\varepsilon}_2 = \partial \Omega_2 \cup S^{\varepsilon}$. The
domains $O_{\varepsilon}$ and $\Omega^{\varepsilon}_2 $ represent,
respectively, the solid and fluid parts of the porous medium
$\Omega$. For simplicity, we suppose $L/\varepsilon , H/\varepsilon , h/\varepsilon \in \mathbb{N}$.

We set $\Sigma =(0,L) \times \{ 0\} $, $\Omega_1 = (0,L)\times (0,h)$ and $\O = (0,L) \times (-H,h)$.
Furthermore, let $\O^\ep = \Oe_2 \cup \Sigma \cup \O_1 $.

\begin{figure}
\centerline{\includegraphics[width=.6\textwidth]{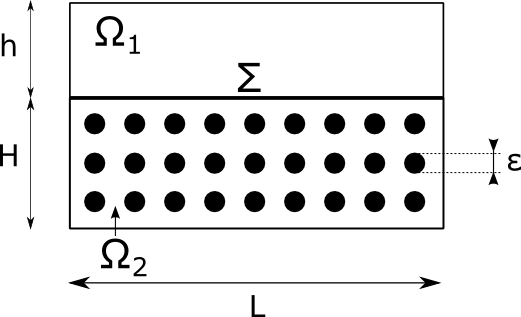}}
\caption{The geometry}\label{afoto}
\end{figure}


 A very important property of the porous media is the following variant of
Poincar\'e's inequality:

\begin{lemma} \label{L1} (see e.g. \cite{SP80})
Let $\varphi \in V(\Oe_2 ) = \{ \varphi \in H^1 (\Oe_2 ) \ | \varphi =0
 \; \hbox{ on } {S}^\ep \} .$
Then, it holds
\begin{gather}
 \Vert \varphi \Vert _{L^2 (\Sigma )} \leq C \ep^{1/2} \Vert  \nabla_x
\varphi \Vert _{L^2 (\Oe_2 )^2}, \label{Poinc1} \\
 \Vert \varphi \Vert _{L^2 (\Oe_2  )} \leq C \ep \Vert  \nabla_x
\varphi \Vert _{L^2 (\Oe_2 )^2}.\label{Poinc2}
\end{gather}
\end{lemma}

\subsection{The microscopic equations}
 \vskip0pt

Having
defined the geometrical structure of the porous medium, we precise
the flow problem. Here we consider the  slow viscous incompressible
flow of a single fluid through a porous medium. We suppose the
no-slip condition at the boundaries of the pores (i.e., a rigid porous medium). Then, we describe
it by the following non-dimensional steady Stokes system in $\Omega^\varepsilon $
(the fluid part of the porous medium $\Omega$):
 \begin{gather} -
 \Delta \mathbf{v}^\ep + \nabla p^{\ep} = \mathbf{f} \qquad \hbox{ in } \quad \Oe
\label{1.3} \\ \div \, \ve = 0 \qquad \hbox{ in } \quad \Oe , \qquad \int_{\O_1} p^\ep \ dx =0,
\label{1.4} \\ \ve  =0 \quad \hbox{on } \quad \p \Oe \setminus \bigg(  \{ x_1 = 0 \} \cup \{ x_1 =L \} \bigg)
, \qquad \{ \ve , p^\ep \} \quad \hbox{ is }
L-\hbox{periodic in } \; x_1 . \label{1.5} \end{gather}
  Here the non-dimensional  $ \mathbf{f}$ stands for the effects of external forces or an
injection at the boundary or a given pressure drop, and it corresponds to the physical forcing term
multiplied by the ratio between Reynolds' number and Froude's number squared. $\ve $
denotes the non-dimensional velocity and  $p^\ep$ is the non-dimensional pressure. The non-constant force $f$ corresponds,  e.g., to a non-constant pressure drop or to injection profiles which are not parabolic.  \vskip0pt Let
\begin{equation}\label{WE}
   W^\ep = \{ \mathbf{z} \in H^1 (\Oe )^2 , \;
\mathbf{z}=0 \; \hbox{ on } \; \p \Oe \setminus  \bigg(  \{ x_1 = 0 \} \cup \{ x_1 =L \} \bigg) \; \hbox{ and }
\; \mathbf{z} \; \hbox{ is } \; L-\hbox{periodic in } x_1\} .
\end{equation}
The
variational form of the problem (\ref{1.3})-(\ref{1.5}) reads:
\vskip0pt $\qquad \qquad $ Find $\ve \in W^\ep  $, div $\ve =0$
in $\Oe$ and $p^\ep \in L^2 (\Oe )$ such that
\begin{equation}\label{1.6}
    \int_{\Oe} \nabla \ve \nabla \varphi \, dx  - \int_{\Oe} p^\ep \hbox{ div }
\varphi \, dx = \int_{\Oe} \mathbf{f} \varphi \, dx \qquad \forall \varphi
\in W^\ep .
\end{equation}
 Then for $\mathbf{f}\in C^\infty ({\overline \O})^2 $, the elementary elliptic
variational theory gives the existence of the unique velocity
field  $\ve \in W^\ep$ ,
 div $\ve =0$ in $\Oe$, which solves (\ref{1.6}) for every $ \varphi \in W^\ep
 , $ div $\varphi =0$ in $\Oe$. The construction of the pressure field
goes through De Rham's theorem (see e.g. book \cite{Tem}).

\subsection{Main result}
 \vskip0pt

 We start by introducing the effective problems in $\O_1$ (the unconfined fluid part) and $\O_2$:

Find a velocity field
$u^{0}$ and a pressure field $p^{eff}$ such that
\begin{gather}
- \triangle \mathbf{u}^{eff}  + \nabla p^{eff} = \mathbf{f} \qquad \hbox{ in } \O_1 ,\label{4.91}\\
\div \ \mathbf{u}^{eff} = 0 \qquad \hbox{ in } \O_1 ,  \qquad \int_{\O_1} p^{eff} \ dx =0,\label{4.92}\\ \mathbf{u}^{eff} =
0  \qquad  \hbox{ on } (0,L) \times  \{h\}  
; \quad \mathbf{u}^{eff} \; \hbox{ and }  \;  p^{eff} \quad  \hbox{ are } \;
L-\hbox{periodic in} \quad x_1, \label{4.94}\\ u^{eff}_2 = 0  \qquad
\hbox{ and } \quad u^{eff}_1 + \ep C^{bl}_1 \frac{\p u^{eff}_1 }{ \p
x_2 } =0 \quad \hbox{ on } \quad \Sigma  . \label{4.95}\end{gather}
We note that the second boundary condition in (\ref{4.95}) is the {\bf  law by Beavers and Joseph} from \cite{BJ}. The constant $C^{bl}_1$ is strictly negative and calculated through (\ref{4.15}), from the viscous boundary layer described in Appendix 2.\vskip1pt
 Problem (\ref{4.91})-(\ref{4.95}) has a unique solution, which in the case of Poiseuille flows (i.e. when $\mathbf{f} = -\displaystyle \frac{p_b - p_0}{L} \mathbf{e}^1 $) reads
\begin{equation}\label{4.96}
\mathbf{u}^{eff}_{pois} = \bigg( { \frac{p_b - p_0 }{ 2 L   }} \bigl( x_2  - {
\frac{\ep C^{bl}_1 h}{ h - \ep C^{bl}_1 }} \bigr) (x_2 - h) , 0
\bigg) \; \hbox{  for } \;  0\leq x_2 \leq h ; \quad p^{eff}= 0  \;
\hbox{  for } \; 0\leq x_1 \leq L .
\end{equation}
 The {\bf effective mass flow rate}
through the channel is then
\begin{gather}
 M^{eff} =\int_{\O_1} u^{eff}_1  \ d x, \label{4.97}\\
\hskip-2cm \mbox{ which for the Poiseuille flow reads} \qquad
   M^{eff}_{pois} =  - \frac{p_b - p_0 }{ 12    }
h^3 \frac{h- 4 \ep C^{bl}_1 }{ h - \ep C^{bl}_1 }  .
\end{gather}
\begin{theorem} \label{P4.19} Let us suppose $\mathbf{f} \in C^\infty ({\overline \O})^2$ and $L$-periodic with respect to $x_1$. For $\{ \ve , p^\ep \} $ given by (\ref{1.3})-(\ref{1.5}) and $\{ \mathbf{u}^{eff} , p^{eff} \} $ by (\ref{4.91})-(\ref{4.95}). It holds
\begin{gather}
 \Vert  \ve - \mathbf{u}^{eff}  \Vert_{L^{2} (\O_1 )^2} + \vert  M^\ep - M^{eff}  \vert \leq
C\ep^{3/2} \label{4.100A} \\
 \Vert  \ve - \mathbf{u}^{eff}  \Vert_{H^{ 1/2 } (\O_1 )^2} +\| p^\ep - p^{eff}   \|_{L^{1}(\O_1)}
+\|  \nabla (\ve - \mathbf{u}^{eff} ) \|_{L^{1}(\O_1)^4}
 +\notag \\
    \| | x_2 |^{1/2} \nabla (\ve - \mathbf{u}^{eff} ) \|_{L^{2}(\O_1)^4}
  + \| | x_2 |^{1/2} (p^\ep - p^{eff} )  \|_{L^{2}(\O_1)^2} \leq
C\ep , \label{4.100}\end{gather}
with $M^{eff}$ defined in (\ref{4.97}).
\end{theorem}
\noindent Next, we study the situation in the porous medium $\O_2$.


\begin{theorem} \label{Limitpress}  Let the permeability tensor $K$ be given by (\ref{1.58}).
The effective porous media pressure ${\tilde p}^0 $  is the $L-$ periodic in $x_1$ function satisfying
\begin{gather}
    \mbox{ div } \bigg( K  (\mathbf{f} (x) - \nabla  {\tilde p}^0 )\bigg) =0\; \mbox{ in } \; \O_2  \label{Presspm} \\
 {\tilde p}^0  = p^{eff} + C^{bl}_\omega  \frac{\p u^{eff}_1}{ \p x_2 } (x_1 , 0 )  \; \mbox{ on } \; \Sigma; \quad K  (\mathbf{f} (x) - \nabla  {\tilde p}^0 ) |_{\{ x_2 =-H \} } \cdot \mathbf{e}^2 =0,  \label{Presspm2}
\end{gather}
with $\mathbf{u}^{eff}$ being the solution to the problem (\ref{4.91})-(\ref{4.95}) and $C_\omega^{bl}$ being the pressure stabilization constant defined by (\ref{4.19}).
In addition
we have
\begin{gather}
    \frac{1 }{\ep^2}  \ve  - K  (\mathbf{f} - \nabla  {\tilde p}^0 ) \rightharpoonup 0 \; \mbox{ weakly in } \; L^2 ((0,L)\times (-H, -\delta ) )^2 , \quad \mbox{ as } \; \ep \to 0, \quad \forall \delta >0;   \label{ConcDarcy}\\
    p^\ep - {\tilde p}^0 \to 0  \; \mbox{ strongly in } \; L^2 (\O_2 ) , \quad \mbox{ as } \; \ep \to 0; \label{ConcPression} \\
    || p^\ep - p^{eff} ||_{H^{-1/2} (\Sigma)} \leq C \sqrt{\ep} .\label{EstPressBdry}
\end{gather}
\end{theorem}
 \begin{remark} If we include the vicinity of $\Sigma$ the velocity $\ve$ has to be corrected by a boundary layer term $\beta^{bl, \ep} (x) =\ep \beta^{bl} (x/\ep )$, defined through (\ref{BJ4.2})-(\ref{4.6}), and the convergence result (\ref{ConcDarcy}) reads
\begin{equation}\label{ConvGlo}
    \frac{1 }{\ep^2} \big( \ve + \beta^{bl , \ep}  \frac{\p u^{eff}_1}{ \p x_2 } (x_1 , 0 ) \big) - K  (\mathbf{f} (x) - \nabla  {\tilde p}^0 ) \rightharpoonup 0 \; \mbox{ weakly in } \; L^2 (\O_2 )^2 , \quad \mbox{ as } \; \ep \to 0.
\end{equation}
\end{remark}
\begin{remark}
\label{R4.21}  Let $\O_{a\ep} = (0,L)\times ( a\ep , h)$
for $a<0$ and let $\{ u^{a, eff} , p^{a, eff } \}$ be a solution
for (\ref{4.91})-(\ref{4.95}) in $\O_{a\ep}$, with (\ref{4.95}) replaced by
\begin{equation}\label{4.103}
   u^{a, eff}_2 = 0  \qquad \hbox{ and } \quad u^{a, eff}_1 + \ep
C^{a, bl}_1
 \frac{\p u^{a, eff}_1 }{ \p x_2 } =0 \quad \hbox{ on } \quad \Sigma_{a} =
(0,b) \times {a\ep}.
\end{equation}
Problem (\ref{4.91})-(\ref{4.94}), (\ref{4.103}) has a unique smooth solution $\{ u^{a,
eff} , p^{a, eff } \}$, its derivatives are bounded independently of $\ep$  and, by (\ref{4.101}),
$C^{a, bl}_1 = C^{bl}_1 - a$. Then a simple calculation gives
\begin{gather} 0= u^{a, eff}_1 (x_1 , \ep a )+ \ep
C^{a, bl}_1
 \frac{\p u^{a, eff}_1 }{ \p x_2 } (x_1 , \ep a) = u^{a, eff}_1 (x_1 , 0 )+ \ep
C^{bl}_1
 \frac{\p u^{a, eff}_1 }{ \p x_2 } (x_1 , 0) +\notag \\
 \frac{(\ep a )^2}{2} (\frac{\p^2 u^{a, eff}_1 }{ \p x_2^2 } (x_1 , \xi_1 ) + \frac{\p^2 u^{a, eff}_1 }{ \p x_2^2 } (x_1 , \xi_2 ) ) ,\quad
\mbox{ for } \quad \xi_1 , \xi_2 \in (0, \ep a).\notag
\end{gather} Therefore, a perturbation of the interface position for an
$O(\ep)$ implies a perturbation in the solution of $O(\ep^2)$ in $ H^k (\O_1)$.
Consequently, there is a freedom in fixing  position of $\Sigma$.
It influences the result only at the next order of the asymptotic
expansion.

{The physical permeability $ K_{phys} $ is proportional to $\varepsilon^2$. Our result on the influence of the interface position  on the effective slip  is in agreement with the observation of Kaviany  in \cite{K95},  pages 79-83. In fact, it has been noticed by Larson and   Higdon in \cite{LH86} that changes of $O(1)$ in the slip coefficients are  possible, after the change of order $O(\sqrt{ K_{phys}) }$ of the interface  position.  Therefore, the  exact position of $\Sigma$ does not pose 
problems, since it influences the solution only at order $O(K_{phys})$.}\end{remark}
\section{Law by Beavers and Joseph}\label{BJSec}

  \vskip0pt In this section we extend the
justification of the law (\ref{AUX}) from \cite{JM00} to the case with a general body force. Our boundary
conditions are simpler from those of the experiment from \cite{BJ} and  we consider the 2D
Stokes system.  The Beavers and Joseph setting could be reduced to our setting if $\O$  is sufficiently
long in $x_1$ direction. Then we may assume the periodic boundary conditions at
inlet/outlet boundary and the flow is  governed by a force coming
from the pressure drop and is equal to $\displaystyle \frac{p_b -p_0
}{b} \mathbf{e}^1$. We assume a non-constant force, which can describe a larger class of the problems.\subsection{The impermeable interface approximation}\label{impp}
Intuitively, the main flow is in the unconfined domain $\O_1$. Following the approach from \cite{JM00} we study the problem
\begin{gather}
- \triangle \mathbf{v}^0  + \nabla p^0 = \mathbf{f}
\qquad \hbox{ in } \O_1 ,\label{4.37}\\
\div \  \mathbf{v}^0 = 0 \qquad \hbox{ in } \O_1 ,\label{4.38}\\
\mathbf{v}^0 = 0  \qquad  \hbox{ on } \p   \O_1  \setminus \bigg(  \{ x_1 = 0 \} \cup \{ x_1 =L \} \bigg) \quad
,\label{4.39}\\ \{ \mathbf{v}^0 , p^0 \}    \qquad  \hbox{ is } \;
L-\hbox{periodic in } \; x_1 \label{4.40}\end{gather}
 Problem (\ref{4.37})-(\ref{4.40}) has
a unique solution $\{ \mathbf{v}^0 , p^0 \} \in H^1 (\O_1)^2 \times L^2_0
(\O_1) $ (see e.g. book \cite{Tem}). In fact this solution is $C^\infty$ for $\mathbf{f} \in C^\infty$.
Therefore, for the lowest order
approximation $\{ \mathbf{v}^0 , p^0 \}$ we impose on the interface the no-slip
condition
\begin{equation}\label{BJ3}
    \mathbf{v}^0 = 0 \qquad \mbox{on} \quad \Sigma.
\end{equation}
 \vskip1pt We observe that in the Beavers and Joseph setting $ \mathbf{f} =-\displaystyle \frac{p_b -p_0}{  L} \mathbf{e}^1$ and the unique
solution for this problem in $H^1 (\O_1)^2 \times L^2_0 (\O_1) $ is
the classic Poiseuille flow in $\O_1$, satisfying the no-slip
conditions at $\Sigma$. It is given by
\begin{equation}\label{4.41}
\mathbf{v}^0 = \bigg( {\displaystyle \frac{p_b - p_0 }{ 2 L   }} x_2 (x_2 -
h) , 0\bigg)
\; \hbox{  for} \quad 0\leq x_2 \leq h ; \qquad 
 p^0 = 0 \; \hbox{  for }\quad 0\leq x_1 \leq L
\end{equation}
(see \cite{JM00} and \cite{JaegMik09} for further details). 
\vskip2pt
 We extend  $\mathbf{v}^0$  to $\O_2$ by setting
$v^0 =0$ for $-H \leq x_2 < 0$. For
$p^0$ we use a smooth extension to $\O_2$, ${\tilde p}^0$, which we shall precise. The question is in which sense this solution approximates the
solution $\{ \ve , p^\ep \}$ of the original problem
(\ref{1.3})-(\ref{1.5}).

Direct consequence of the weak formulation (\ref{1.6}) is that the difference $\ve - \mathbf{v}^0$ satisfies the following variational equation:
\begin{equation}\label{H4.57}
 \int_{\Oe} \nabla (\ve - \mathbf{v}^0 ) \nabla \varphi \ dx -\int_{\Oe}  ( p^\ep - {\tilde p}^0 ) \mbox{ div } \varphi   = \int_{\Sigma}  \frac{\p   v_1^0}{ \p   x_2 }   \varphi_1 \ dS  -  \int_{\Sigma}  [{\tilde p}^0 ]  \varphi_2 \ dS + \int_{\Oe_2}(\mathbf{f} - \nabla {\tilde p}^0 ) \varphi \ dx  , \ \forall \varphi \in   { W}^\ep .
\end{equation}
Taking $\varphi =\ve - \mathbf{v}^0$ in (\ref{H4.57}) and applying
Lemma \ref{L1} leads to the following  result,
proved in \cite{JM00}:
\begin{proposition} \label{P4.14} Let $\{
\ve , p^\ep \}$ be the solution  for (\ref{1.3})-(\ref{1.5}) and $\{ \mathbf{v}^0 , p^0 \}$ defined by (\ref{4.37})-(\ref{4.40}). Then, it holds
\begin{gather}
\sqrt{\ep} \Vert \nabla (\ve - v^0 ) \Vert_{L^2 (\Oe )^4} +
  \frac{1}{\sqrt{\ep}}\Vert \ve  \Vert_{L^2 (\Oe_2 )^2} 
 +\Vert \ve  \Vert_{L^2 (\Sigma )} 
\leq C \ep
 \label{4.52}\end{gather}
\end{proposition}
Furthermore, using estimate (\ref{4.52}) and the notion of  very weak solutions for the Stokes system in $\O_1$, introduced in \cite{CO} (see also Appendix 1),  we conclude the following additional estimates:
\begin{corollary} \label{veryweak1} (see \cite{JM00}) Let $\{
\ve , p^\ep \}$ be the solution  for (\ref{1.3})-(\ref{1.5}) and
$\{ \mathbf{v}^0 , p^0 \}$ defined by (\ref{4.37})-(\ref{4.40}). Then, it holds
\begin{gather}
\sqrt{\ep} \Vert p^\ep  - p^0 \Vert_{L^2 (\O_1 ) }
+ \Vert \ve - v^0 \Vert_{L^2 (\O_1 )^2} \leq C \ep.
 \label{vw1}\end{gather}
\end{corollary}
 This provides the uniform a priori
estimates for $\{ \ve , p^\ep \} $. Moreover, we have found that the viscous flow in $\O_1$ corresponding to an impermeable wall  is an $O(\ep )$ $L^2$-approximation
for $\ve$.  Beavers and Joseph's law should correspond to the next
order velocity correction. Since the Darcy velocity is of order O$(\ep^2 )$ we
justify Saffman's version of the law. \vskip2pt

\subsection{Justification of the law
by Beavers and Joseph} \label{just}

At the interface $\Sigma$ the approximation from Subsection \ref{impp} leads to 
 the shear
stress jump equal to $\displaystyle - \frac{\partial
v_1^{0}}{\partial x_2} |_{\Sigma}$. Contrary to the pressure difference, which could be  easily set to zero by the appropriate choice of ${\tilde p}^0$, the shear stress jump requires construction of the corresponding boundary layer.
For the intuitive argument how to obtain the shear stress jump correction using the natural stretching
variable  $\displaystyle
y=\frac{x}{\ep}$, we refer to the paper \cite{JaegMik09}, page 503.
In the present paper we present the rigorous construction, based on the Navier boundary layer and following the scheme originally used in \cite{JM00}.

Let $\{ \beta^{bl} , \omega^{bl} \}$ be the boundary layer given by
(\ref{BJ4.2})-(\ref{4.6}).\vskip0pt Now we set
\begin{equation}\label{4.59}
   \beta^{ bl, \ep} (x)= \ep \beta^{ bl} (\frac{x}{ \ep } ) \qquad
\hbox{and} \qquad \omega^{bl,\ep} (x)= \omega^{bl} (\frac{x}{ \ep }
), \quad x\in \Oe,
\end{equation}
  $\beta^{ bl ,\ep } $ is extended by
zero to  $\O \setminus \Oe $. Let $H$ be Heaviside's function. Then
for every $q\geq 1$ we obtain
\begin{equation}\label{4.60}
\frac{1}{\ep} \Vert \beta^{ bl ,\ep} - \varepsilon (C^{bl}_1 , 0) H(x_2) \Vert
_{L^q(\O)^2} 
+ \Vert \omega^{bl,\ep} -C^{bl}_\omega H(x_2) \Vert _{L^q(\Oe)} +
\Vert \nabla \beta^{ bl , \ep } \Vert _{L^q(\O_1 \cup \Sigma \cup
\O_2 )^4} =C \ep^{1/q} .
\end{equation}
Hence, our boundary layer is not concentrated around the interface and
there are some stabilization constants. We will see that these
constants are closely linked to our effective interface law.
\vskip1pt
  As in \cite{JaMi2} stabilization of $\beta^{bl , \ep }$
towards a nonzero constant velocity $\ep \big( C^{bl}_1 , 0\big) $,
at the upper boundary,  generates a counterflow. It is given by the following
 Stokes system  in $\Omega_1$:
\begin{gather}
- \triangle \mathbf{z}^\sigma  + \nabla p^\sigma = 0
\qquad \hbox{ in } \O_1 ,\label{4.37Couette}\\
\div \  \mathbf{z}^\sigma = 0 \qquad \hbox{ in } \O_1 ,\label{4.38Couette}\\
\mathbf{z}^\sigma = 0  \quad  \hbox{ on }  \{ x_2 = h \} \quad \mbox{and} \; \mathbf{z}^\sigma =\frac{\p   v^0_1 }{ \p x_2 }
|_\Sigma  \mathbf{e}^1 \quad  \hbox{ on }  \{ x_2 = 0 \}
,\label{4.39Couette}\\ \{ \mathbf{z}^\sigma , p^\sigma \}    \qquad  \hbox{ is } \;
L-\hbox{periodic in } \; x_1 .\label{4.40Couette}\end{gather}
 In  the setting of the experiment by Beavers and Joseph, $\mathbf{z}^\sigma$ was proportional to the
two dimensional Couette flow $\mathbf{d}={\displaystyle
(1- \frac{x_2 }{ h}})  \mathbf{e}^1 $.\vskip1pt

Now,  after \cite{JaMi2}, we  expected that the approximation for the velocity reads
\begin{gather}
\ve = \mathbf{v}^0  -  (\beta^{ bl , \ep} - \ep ( C^{bl}_1 ,0 ))  \frac{\p   v^0_1 }{ \p x_2 }
|_\Sigma -   \ep  C^{bl}_1  \mathbf{z}^\sigma + O(\ep^2 ), \label{4.66} 
\end{gather}
\vskip2pt  Concerning the pressure,  there are additional complications due to the
stabilization of the boundary layer pressure to $C^{bl}_\omega $,
when $y_2 \to +\infty$. Consequently, 
 $\displaystyle {\omega}^{ bl , \ep} - H(x_2) C^{bl}_\omega  \frac{\p   v^0_1 }{ \p    x_2 } |_\Sigma $ is small
in $\Omega_1$ and we should take into account the pressure stabilization effect.

At the flat
interface $\Sigma$, the normal component of the normal stress
reduces to the pressure field. Subtraction of the stabilization
pressure constant at infinity leads to the pressure jump on
$\Sigma$:
\begin{equation}\label{BJ46}
    [p^\ep]_\Sigma = p^0 (x_1, +0) -  {\tilde p}^0 (x_1, -0) =-  C^{bl}_\omega
  \frac{\p   v^0_1 }{ \p    x_2 } |_\Sigma + O (\ep ) \quad \mbox{ for }
\quad x_1 \in (0,L).
\end{equation}
  Therefore, the pressure approximation is
\begin{equation}\label{BJ47}
p^\ep (x) =p^0 H(x_2) + {\tilde p}^0  H(-x_2 ) - \bigl( {\omega}^{ bl , \ep} (x) - H(x_2)
C^{bl}_\omega \bigr)   \frac{\p   v^0_1 }{ \p    x_2 } |_\Sigma -
  \ep  C^{bl}_1  p^\sigma H(x_2) + O(\ep) .
\end{equation}
\vskip3pt Following the ideas from \cite{JaMi2},  these heuristic
calculations could be made rigorous. Let us define the errors in
velocity and in the pressure:
\begin{gather}
{\cal U}^\ep (x) =  \ve -  \mathbf{v}^0  +( \beta^{ bl , \ep} -\ep   C^{bl}_1 \mathbf{e}^1 H(x_2))  \frac{\p v^0_1
}{ \p   x_2 } |_\Sigma  +\ep   C^{bl}_1 \mathbf{z}^\sigma \label{BJ4.66} \\
{\cal P}^\ep (x) = p^\ep - p^0 H(x_2) - {\tilde p}^0  H(-x_2 ) +\bigl( {\omega}^{ bl , \ep} (x) - H(x_2)
C^{bl}_\omega \bigr)   \frac{\p   v^0_1 }{ \p    x_2 } |_\Sigma +
  \ep  C^{bl}_1  p^\sigma H(x_2) . \label{BJ4.67}
\end{gather}
 \begin{remark} Rigorous argument, showing that ${\cal U}^\ep$ is of order $O(\ep^2)$, allows  justifying  Saffman's
modification  of the Beavers and Joseph law (\ref{AUX}):
On the interface
$\Sigma$ we obtain
$$\frac{\p v^\ep _1}{ \p x_2 } |_\Sigma = \frac{\p v^0 _1}{ \p x_2 }
|_\Sigma - \frac{\p \beta^{bl}_1 }{ \p y_2} |_{\Sigma , y=x/\ep}  +
O(\ep) \quad \hbox{ and } \quad \frac{v^\ep _1}{ \ep } = -
\beta^{bl}_1 ( x_1 / \ep , 0) \frac{\p v^0 _1}{ \p x_2 } |_\Sigma +
O(\ep) . $$ After averaging over $\Sigma$ with respect to $y_1$, we
obtain the 
Saffman version of the law by
Beavers and Joseph
  \begin{equation}\label{BJ}
u^{eff}_1 = -\ep C^{bl}_1 \frac{\p u^{eff}_{1} }{ \p x_2}  \quad
\hbox{ on } \quad \Sigma ,    \end{equation}
 where $u^{eff}_1$ is the average of $v_1^{\ep}$
over the characteristic pore opening at the naturally permeable wall. The
higher order terms are neglected.\end{remark}
For simplicity we denote
$$  \sigma_{12}^0 (x_1 ) = \frac{\p v^0_1
}{ \p   x_2 } |_\Sigma .$$
Then, the variational equation for $\displaystyle ( \beta^{ bl , \ep} -\ep   C^{bl}_1 \mathbf{e}^1 H(x_2))  \frac{\p v^0_1
}{ \p   x_2 } |_\Sigma $ reads
\begin{gather}
    \int_{\Oe} \nabla \bigg( ( \beta^{ bl , \ep} -\ep   C^{bl}_1 \mathbf{e}^1 H(x_2)) \sigma_{12}^0 \bigg) :  \nabla \varphi \ dx - \int_{\Oe} \sigma_{12}^0 \bigl( {\omega}^{ bl , \ep} (x) - H(x_2)
C^{bl}_\omega \bigr) \mbox{ div } \varphi \ dx =\notag \\
 -\int_{\Sigma} \varphi_1   \sigma_{12}^0 \ dS
 - \int_{\Sigma} C^{bl}_\omega  \varphi_2   \sigma_{12}^0 \ dS
- \int_{\Oe} \sum_i \bigg(  \Delta \sigma_{12}^0 ( \beta^{ bl , \ep}_i -\ep   C^{bl}_1 \delta_{1i} H(x_2)) \varphi_i -\notag \\  \p_{x_i} \sigma_{12}^0 ( \omega^{ bl , \ep} -\ep   C^{bl}_\omega) \varphi_i  -2 ( \beta^{ bl , \ep}_i -\ep   C^{bl}_1 \delta_{1i} H(x_2)) \mbox{ div } (\varphi_i \nabla \sigma_{12}^0 ) \bigg) \ dx , \; \forall \varphi \in  { W}^\ep .\label{eqbeta}
\end{gather}
Next, the variational form of (\ref{4.37Couette})-(\ref{4.40Couette}) reads
\begin{gather}
 \int_{\Oe} \nabla  \mathbf{z}^\sigma :  \nabla \varphi \ dx - \int_{\Oe} p^\sigma \mbox{ div } \varphi \ dx =  -\int_{\Sigma} (-\varphi_2 p^\sigma +\varphi \cdot  \frac{\p \mathbf{z}^\sigma }{ \p x_2 })    \ dS , \; \forall \varphi \in  { W}^\ep .
\label{zsigma}
\end{gather}

Now we are ready to write the variational equation for $\{ {\cal U}^\ep , {\cal P}^\ep \}$ and obtain the higher order error estimates as in  \cite{JM00}. Nevertheless, contrary to \cite{JM00}, $  {\cal U}^\ep$ is not divergence free anymore and we need 
more effort to control ${\cal P}^\ep$.
\begin{theorem} \label{T4.17} Let ${\cal U}^\ep$ be defined by (\ref{BJ4.66}) and
${\cal P}^\ep$ by (\ref{BJ4.67}).  Let ${\tilde p}^0$ be a smooth function satisfying the interface condition (\ref{BJ46}).
 Then, the following
estimates hold
\begin{gather}
 \ep \Vert \nabla {\cal P}^\ep  \Vert_{H^{-1} (\Oe)} +
\ep \Vert \nabla {\cal U}^\ep  \Vert_{L^2 (\Oe )^4} + \Vert  {\cal U}^\ep  \Vert_{L^2 (\Oe_2 )^2} +
\ep^{1/2} \Vert  {\cal U}^\ep  \Vert_{L^2 (\Sigma  )^2} 
 \leq C\ep^2
 \label{4.84} 
 \end{gather}
\end{theorem}
\begin{proof}
First we remark that for $y_2 >0$ the mean with respect to $y_1$ of ${\omega}^{ bl } (y)-
C^{bl}_\omega$ is zero. Consequently, the problem
\begin{equation}\label{pressaux}
    \frac{\p \pi^{bl}_\omega }{\p y_1} = {\omega}^{ bl } (y)-
C^{bl}_\omega \quad \forall y_1 \in (0,1); \quad \pi^{bl}_\omega \; \mbox{ is 1-periodic }; \quad \int^1_0 \pi^{bl}_\omega (y_1 , y_2 ) \ dy_1 =0,
\end{equation}
has a unique smooth solution.

Next by subtracting (\ref{eqbeta}) and (\ref{zsigma}) from (\ref{H4.57}) we obtain
\begin{gather}
    \int_{\Oe} \nabla {\cal U}^\ep  :  \nabla \varphi \ dx - \int_{\Oe} {\cal P}^\ep \mbox{ div } \varphi \ dx =\notag \\
\ep  \int_{\Sigma} (-\varphi_2 p^\sigma +\varphi \cdot  \frac{\p \mathbf{z}^\sigma }{ \p x_2 })    \ dS +\int_{\Oe_2}(\mathbf{f} - \nabla {\tilde p}^0 ) \varphi \ dx
- \int_{\Oe} \sum_i \bigg(  \Delta \sigma_{12}^0 ( \beta^{ bl , \ep}_i -\ep   C^{bl}_1 \delta_{1i} H(x_2)) \varphi_i \ dx  -\notag \\  \int_{\Oe_2}  \p_{x_1} \sigma_{12}^0  \omega^{ bl , \ep}  \varphi_1 \ dx  -\int_{\O_1} \ep  \pi^{bl}_\omega (\frac{x}{\ep} ) (\varphi_1 \p_{x_1}^2  \sigma_{12}^0 +\p_{x_1}\varphi_1 \p_{x_1}  \sigma_{12}^0  ) \ dx  +\notag \\  2 \int_{\Oe} ( \beta^{ bl , \ep}_1 -\ep   C^{bl}_1  H(x_2)) (\varphi_1 \p_{x_1}^2  \sigma_{12}^0 +\p_{x_1}\varphi_1 \p_{x_1}  \sigma_{12}^0  ) \ dx , \; \forall \varphi \in  { W}^\ep ,\label{UPsystem} \\
\mbox{ div }  {\cal U}^\ep = ( \beta^{ bl , \ep}_1  -\ep   C^{bl}_1  H(x_2)) \frac{d}{d x_1}\sigma_{12}^0 \quad \mbox{ in } \; \Oe . \label{div1}
\end{gather}
From (\ref{UPsystem}) we find out that
\begin{gather}
| \int_{\Oe} \nabla {\cal U}^\ep  :  \nabla \varphi \ dx - \int_{\Oe} {\cal P}^\ep \mbox{ div } \varphi \ dx | \leq C\ep^{3/2} || \nabla \varphi ||_{L^2 (\Oe )^4} + C\ep || \mathbf{f} - \nabla {\tilde p}^0 ||_{L^2 (\O_2 )^2} || \nabla \varphi ||_{L^2 (\Oe_2 )^4}
\label{Est1} \\
    \mbox{ and } || \mbox{ div }  {\cal U}^\ep ||_{L^2 (\Oe )^2} \leq C\ep^{3/2} .\label{divest1}
\end{gather}
The size of $\mbox{ div }  {\cal U}^\ep$  does not allow us to obtain the appropriate estimate and we should  diminish it further.

Let $\mathbf{Q}^{ bl}$ be given by (\ref{Div1})-(\ref{Div3}). Furthermore let $\mathbf{Q}^{ bl , \ep} (x) = \ep^2 \mathbf{Q}^{ bl} (x/ \ep )$ and let $\mathbf{w}^{Q}$ be defined by
\begin{equation}\label{Eqwkj}
    \left\{
      \begin{array}{ll}
      \displaystyle
        \Delta \mathbf{w}^{Q} - \nabla p^{Q}  = 0   \;  \hbox{ in }
         \hskip1pt  \O_1  ;&\\ \noalign{\vskip+4mm}
       \hbox{div\  } \mathbf{w}^{Q} =\displaystyle \frac{1}{| \O_1 |} \int_\Sigma \frac{d}{d x_1}\sigma_{12}^0 \ dS =0 \quad \hbox{ in } \hskip3pt \O_1
       ;&\\ \\
        \mathbf{w}^{Q} =-\displaystyle \frac{d}{d x_1}\sigma_{12}^0 \mathbf{e}^2  \,  \hbox{ on } \hskip3pt \Sigma  , \quad \mathbf{w}^{Q}=0 \,  \hbox{ on } \hskip3pt \{ x_2 =h\} ; &\\ \\
       \{ \mathbf{w}^{Q}  , p^{Q} \}   \;  \hbox{is L-periodic in } \; x_1 . &
      \end{array}
    \right.
\end{equation}
We introduce the following error functions, where the compressibility effects are reduced to the next order:
\begin{gather}
 {\cal U}^\ep (x) =  {\cal U}^\ep_0 (x)  + \mathbf{Q}^{ bl , \ep} (x) \frac{d}{d x_1}\sigma_{12}^0 +
\ep^2 H(x_2)  ( \int_{Z_{BL}} ( C^{bl}_1 H(y_2) - \beta^{bl}_1 (y)) \ dy  )
\mathbf{w}^{Q}  ,  \label{Eq77}  \\
 {\cal P}^\ep (x) ={\cal P}^\ep_0 (x,t) +\ep^2 H(x_2)   (\int_{Z_{BL}} ( C^{bl}_1 H(y_2) - \beta^{bl}_1 (y)) \ dy  ) 
p^{Q}    , \label{Eq78} \\
\mbox{ div }  {\cal U}^\ep_0  = - {Q}^{ bl , \ep}_1 (x) \ \frac{d^2}{d x_1^2} \sigma_{12}^0 \quad \mbox{ in } \; \Oe . \label{divEq1}
\end{gather}
Then ${\cal U}^\ep_0  \in    { W}^\ep$ and $ || $ div ${\cal U}^\ep_0 ||_{L^2 (\Oe )^4} \leq C\ep^{5/2}$.
Next, we construct a function $\Phi^{1, \ep} \in H^1 (\O_1 )^2$ such that
 \begin{equation}\label{EqPhi1}
    \left\{
      \begin{array}{ll}
       \hbox{div\  }  \Phi^{1, \ep} =\displaystyle - {Q}^{ bl , \ep}_1 (x) \ \frac{d^2}{d x_1^2} \sigma_{12}^0 \quad \hbox{ in } \hskip3pt \O_1 ,
       ;&\\ \\
         \Phi^{1, \ep}  =\displaystyle \frac{\mathbf{e}^2 }{| \Sigma |} \int_{\O_1}  {Q}^{ bl , \ep}_1 (x) \frac{d^2}{d x_1^2}\sigma_{12}^0 \ dx \,  \hbox{ on } \hskip3pt \Sigma  , \quad \Phi^{1, \ep} =0 \,  \hbox{ on } \hskip3pt \{ x_2 =h\} , &\\ \\
       \Phi^{1, \ep}   \;  \hbox{is L-periodic in } \; x_1 . &
      \end{array}
    \right.
\end{equation}
We note that $\displaystyle || \Phi^{1, \ep} ||_{H^1 (\O_1 )^2} \leq C\ep^2$. Next we extend ${Q}^{ bl , \ep}$ by zero to the rigid part of the porous medium and choose a function $\Phi^{2, \ep} \in H^1 (\O_2 )^2$ such that
 \begin{equation}\label{EqPhi11}
    \left\{
      \begin{array}{ll}
       \hbox{div\  }  \Phi^{2, \ep} =\displaystyle - {Q}^{ bl , \ep}_1 (x) \ \frac{d^2}{d x_1^2} \sigma_{12}^0 \quad \hbox{ in } \hskip3pt \O_2 ,
       &\\ \\
         \Phi^{2, \ep}  =-\displaystyle \frac{\mathbf{e}^2 }{| \Sigma |} \int_{\O_2}  {Q}^{ bl , \ep}_1 (x) \frac{d^2}{d x_1^2}\sigma_{12}^0 \ dx \,  \hbox{ on } \hskip3pt \Sigma  , \quad \Phi^{2, \ep} =0 \,  \hbox{ on } \hskip3pt \{ x_2 =-H\} , &\\ \\
       \Phi^{2, \ep}   \;  \hbox{is L-periodic in } \; x_1 . &
      \end{array}
    \right.
\end{equation}
We note that $\displaystyle \Phi^{1, \ep}  = \Phi^{2, \ep} $ on $\Sigma$ and $\displaystyle || \Phi^{2, \ep} ||_{H^1 (\O_1 )^2} \leq C\ep^2$. Let $
X_2= \{ \mathbf{z} \in H^1 (\O_2 )^2 , \;
\mathbf{z}=0 \; \hbox{ on } \; \{ x_1 =L \} $ and $
\; \mathbf{z} \; \hbox{ is } \; L-\hbox{periodic in } x_1\} \quad \mbox{and} \quad X_2^\ep =\{ \mathbf{z} \in X_2 , \;  \mathbf{z}=0 \; \hbox{ on } \; \p \Oe_2 \setminus \p \O_2 \}.$ In the seminal paper  \cite{Ta1980} Tartar constructed a continuous linear restriction operator   operator
$R_{\ep} \in {\cal L} \big( X_2 ,\, X_2^\ep \big)$, such that
  \begin{gather}
\hbox{ div } ( R_{\ep} \varphi ) = \hbox{ div }  \varphi  + \sum_{k\in T_\ep}
{1\over \vert Y^\ep_{F_k} \vert } {\displaystyle
\chi}_{Y^\ep_{F_k} } \int_{Y^\ep_{S_k} } \hbox{ div } \varphi \,
dx , \qquad \forall \varphi \in X_2 \label{1.30} \\
    \Vert R_\ep \varphi
\Vert_{L^2 (\Oe_2)^2}\leq C\big\{\ep \Vert \nabla \varphi \Vert_{L^2 (\O_2)^4
}+  \Vert  \varphi  \Vert_{L^2 (\O_2)^2}\big\}  ,
\quad \forall \varphi \in X_2 \label{1.32} \\
    \Vert \nabla ( R_\ep
\varphi ) \Vert_{L^2 (\Oe)^{4}}\leq 
    {C\over \ep }
\big\{\ep \Vert \nabla \varphi \Vert_{L^2 (\O_2)^4 }+  \Vert  \varphi  \Vert_{L^2 (\O_2)^2}\big\} , \quad \forall \varphi \in
X_2 .\label{1.33} \end{gather}
Furthermore, $\varphi = R_\ep
\varphi $ on $\Sigma$. For more details we refer also to \cite{All97} and \cite{MIK00}.
This construction allows us to work with the divergence free velocity error function $\mathcal{\overline U}^\ep$ given by
\begin{equation}\label{Dercor}
\mathcal{\overline U}^\ep = {\cal U}^\ep_0 - H(x_2) \Phi^{1, \ep} - H(-x_2 ) R_\ep \Phi^{2, \ep}
\end{equation}

Now we write the analogue of   the variational equation (\ref{UPsystem}) for $\{ \mathcal{\overline U}^\ep , {\cal P}^\ep_0 \}$  and, since $\displaystyle || \nabla R_\ep \Phi^{2, \ep} ||_{L^2 (\O_2 )^4} $ $\leq C\ep  $. We
find out that the leading order force term is of the same order as in the estimate (\ref{Est1}).
Now we test the analogue of   variational equation (\ref{UPsystem}) for $\{ \mathcal{\overline U}^\ep , {\cal P}^\ep_0 \}$ with $\varphi =\mathcal{\overline U}^\ep  $
 to obtain
\begin{equation}\label{Uest}
   || \nabla \mathcal{\overline U}^\ep ||_{L^2 (\Oe )^4}   \leq C\ep .
\end{equation}
We remark that $\mathcal{\overline U}^\ep$ differs from $\mathcal{ U}^\ep$ for $O(\ep^2 )$ in $L^2$-norm and for $O(\ep )$ in $H^1$-norm . Therefore (\ref{Uest}) gives us the middle part of the estimate
(\ref{4.84}). In  what concerns the $L^2 (\Sigma)$ norm of $\mathcal{ U}^\ep$, it follows by using (\ref{Poinc1}).
Remaining pressure estimate follows  easily from the weak formulation and the estimates on $\mathcal{ U}^\ep$.
\end{proof}
Next we use Theorem \ref{T4.17} and the results on the Stokes system with $L^2 -$ boundary values from \cite{FKV} and \cite{BZ} to conclude the following result:
\begin{corollary} \label{T4.17c} Let ${\cal U}^\ep $ be defined by (\ref{BJ4.66}) and
${\cal P}^\ep$ by (\ref{BJ4.67}).  Let ${\tilde p}^0$ be a smooth function satisfying the interface condition (\ref{BJ46}).
 Then,  the following
estimate holds
\begin{gather}
\sqrt{\ep} \Vert  {\cal P}^\ep  \Vert_{L^2 (\O_1
)} +
\Vert  {\cal U}^\ep  \Vert_{H^{1/2} (\O_1 )^2} \leq C\ep^{3/2} .
 \label{4.84C} 
 \end{gather}
\end{corollary}
Now we introduce the effective flow equations in $\O_1$ through the
 boundary value problem (\ref{4.91})-(\ref{4.95}), containing the slip condition of Beavers and Joseph.
  Since our expansion is performed using the solution $\{ \mathbf{v}^0 , p^0 \}$ of the problem
(\ref{4.37})-(\ref{4.40}), we need to know the relationship between the solutions to these two boundary value problems.

\begin{proposition}\label{proxio} Let $\mathbf{f} \in C^\infty ({\overline \O}_1)^2$ and $L$-periodic in $x_1$. Let $\{ \mathbf{u}^{eff} , p^{eff} \} $ be the solution of the problem (\ref{4.91})-(\ref{4.95}),  $\{ \mathbf{v}^0 , p^0 \}$ of the problem (\ref{4.37})-(\ref{4.40}) and $\{ \mathbf{z}^{\sigma} , p^{\sigma} \} $  of the problem (\ref{4.37Couette})-(\ref{4.40Couette}). Then we have
\begin{gather}|| \mathbf{u}^{eff} -\mathbf{v}^0 ||_{H^k (\O_1 )^2} + || p^{eff} -p^0 ||_{H^{k-1} (\O_1 )}\leq C\ep , \quad \forall k\in \mathbb{N} ; \label{Prox1} \\
|| \mathbf{u}^{eff} -\mathbf{v}^0 + \ep C^{bl}_1 \mathbf{z}^{\sigma} ||_{H^k (\O_1 )^2} + || p^{eff} -p^0 +  \ep C^{bl}_1 p^{\sigma} ||_{H^{k-1} (\O_1 )}\leq C\ep^2 , \quad \forall k\in \mathbb{N} . \label{Prox2}
\end{gather}
\end{proposition}
\begin{proof} The elliptic regularity for the Stokes operator (see e.g. \cite{Tem}) gives $C^\infty$ regularity for the functions $ \{ \mathbf{u}^{eff} , p^{eff} \} $,  $\{ \mathbf{v}^0 , p^0 \}$ and $\{ \mathbf{z}^{\sigma} , p^{\sigma} \} $. It is easy to see that $ \{ \mathbf{u}^{eff} , p^{eff} \} $ is bounded in $H^k (\O_1)^4$, independently of $\ep$, for every integer $k$.\vskip0pt Let $\mathbf{U}=  \mathbf{u}^{eff} - \mathbf{v}^0$ and $P= p^{eff} -p^0$. Then for every $\varphi \in {\cal V} = \{ \varphi \in H^1 (\O_1 )^2 | \; \varphi $  is $L$-periodic in $x_1 , \; \varphi =0 $ on $\{ x_2 =h \} , \; \varphi_2 =0$ on $\Sigma \}$ we obtain
\begin{gather}\int_{\O_1} \nabla \mathbf{U} : \nabla \varphi \ dx - \int_{\O_1} P \mbox{ div } \varphi \ dx - \frac{1}{\ep C^{bl}_1} \int_{\Sigma} U_1 \varphi_1 \ dS = - \int_{\Sigma} \sigma^0_{12} \varphi_1 \ dS .\label{Eqprox1}
\end{gather}
Using $\varphi = \mathbf{U}$ as a test function yields
\begin{equation}\label{Estprox1}
    \left\{
      \begin{array}{ll}
        \displaystyle || \mathbf{U}||_{H^1 (\O_1 )^2} + \frac{1}{\sqrt{\ep}} || U_1  ||_{L^2 (\Sigma )}\leq C\sqrt{\ep} , &  \\
      \displaystyle || P||_{L^2 (\O_1 )} \leq C\sqrt{\ep}   . &
      \end{array}
    \right.
\end{equation}
Differentiating the equations with respect to $x_1$ leads to the estimate
\begin{equation}\label{Estprox2}
    \left\{
      \begin{array}{ll}
        \displaystyle || \frac{\p \mathbf{U}}{\p x_1}||_{H^1 (\O_1 )^2} + \frac{1}{\sqrt{\ep}} || \frac{\p  U_1  }{\p x_1} ||_{L^2 (\Sigma )}\leq C\sqrt{\ep} , &  \\
      \displaystyle || \frac{\p P}{\p x_1}||_{L^2 (\O_1 )} \leq C\sqrt{\ep}   . &
      \end{array}
    \right.
\end{equation}
Since $\displaystyle \frac{\p  U_2  }{\p x_1} =0$ on $\Sigma$, we have for the velocity trace $\mathbf{U} \in H^1 (\Sigma )^2$ and its norm is smaller than $C\ep$. Using \cite{FKV} and \cite{BZ} we obtain that
\begin{equation}\label{Estprox3}
    || \mathbf{U}||_{H^{3/2} (\O_1 )^2} +  ||  P ||_{H^{1/2} (\O_1 )} \leq C\ep .
\end{equation}
After bootstrapping, we conclude that the estimate (\ref{Prox1}) holds true. \vskip2pt
Using corrections $\mathbf{U}^1=  \mathbf{u}^{eff} - \mathbf{v}^0 + \ep C^{bl}_1 \mathbf{z}^{\sigma}$ and $P^1= p^{eff} -p^0 + \ep C^{bl}_1 p^{\sigma}$, for every $\varphi \in {\cal V} = \{ \varphi \in H^1 (\O_1 )^2 | \; \varphi $  is $L$-periodic in $x_1 , \; \varphi =0 $ on $\{ x_2 =h \} , \; \varphi_2 =0$ on $\Sigma \}$ we obtain
\begin{gather}\int_{\O_1} \nabla \mathbf{U}^1 : \nabla \varphi \ dx - \int_{\O_1} P^1 \mbox{ div } \varphi \ dx - \frac{1}{\ep C^{bl}_1} \int_{\Sigma} U_1^1 \varphi_1 \ dS = \ep \int_{\Sigma} g \varphi_1 \ dS ,\label{Eqprox2}
\end{gather}
where $g=\displaystyle -C_1^{bl} \frac{\p z^\sigma_1}{\p x_2} |_\Sigma \in C^\infty ({\overline \Sigma })$ is uniformly bounded with respect to $\ep$. Repeating the argument used in the first part of the proof to  $\{ \mathbf{U}^1 , P^1 \}$ yields the estimate (\ref{Prox2}).
\end{proof}

\begin{proof} ({\bf of  Theorem \ref{P4.19} }) We remark that on $\Sigma$
\begin{gather}
    \ve - \mathbf{u}^{eff} = {\cal U}^\ep - (\beta^{ bl , \ep} - \ep ( C^{bl}_1 ,0 ))  \frac{\p   v^0_1 }{ \p x_2 } (x_1 ,0).  \label{Diff}
\end{gather}
Now Theorem \ref{T4.17}, Corollary \ref{T4.17c} and  Propositions \ref{D2} and \ref{D3} from the Appendix 1 imply the desired result.
\end{proof}
\subsection{Justification of the interface pressure jump law and the effective equations in the porous medium} \label{justpress}

 We
have already seen that, after extension by zero to the rigid part,
the velocity ${\cal U}^\ep$ satisfies the {\it a priori} estimates
(\ref{4.84}), (\ref{4.84C}), with $\Oe$ replaced by $\O$. Furthermore,
it would be more comfortable to work with the pressure field
${\cal P}^\ep$ defined on $\O$. Following the approach from \cite{LA}, we
define the pressure extension ${\tilde {\cal P}}^\ep$ by
\begin{equation}\label{1.27P}
 {\tilde {\cal P}}^\ep =
\left\{%
\begin{array}{ll}
    { {\cal P}}^\ep &\hbox{ in } \Omega^\ep  \\
    \frac{1}{ \vert
Y^\ep_{F_i}\vert } \, \int_{Y^\ep_{F_i}}\, { {\cal P}}^\ep & \hbox{ in the }
\, Y^\ep_{S_i}\, \hbox{ for each } i  ,  \\
\end{array}%
\right.
\end{equation}
 where $Y^\ep_{F_i}$ is the fluid part of the cell
$Y^\ep_i$. Note that the solid part of the porous medium is a union of
all $Y^\ep_{S_i}$.  Then, following Tartar's results from  \cite{Ta1980}  we have
$$ < \nabla {\tilde {\cal P}}^\ep , \varphi >_\O = < \nabla  {\cal P}^\ep , {\tilde R}_\ep \varphi >_{\O^\ep } , \quad \forall \varphi \in H^1 (\O)^2 , $$
where
\begin{equation}\label{ExtR}
    {\tilde R}_\ep \varphi (x) = \left\{
                                   \begin{array}{ll}
                                     \varphi (x) , & \hbox{for }  \; x\in \O_1 \cup \Sigma ;\\
                                     R_\ep  \varphi (x) , & \hbox{for } \; x\in \Oe_2 .
                                   \end{array}
                                 \right.
\end{equation}
Using the estimate (\ref{4.84}) and properties (\ref{1.30})-(\ref{1.33}) of the restriction operator $R_\ep$, we arrive at
\par
\begin{corollary} \label{C1.6} (a priori estimate for the pressure field in
$\O_2$). Let ${\tilde {\cal P}}^\ep $ be defined by (\ref{1.27P}). Then it
satisfies the estimates
  \begin{gather}  \Vert \nabla {\tilde
{\cal P}}^\ep\Vert_{ W'} \leq C \quad \mbox{ and } \quad \Vert  {\tilde {\cal P}}^\ep
\Vert_{L^2 (\O_2)} \leq C,
    \label{1.29P} \end{gather} where $W=\{
\mathbf{z} \in H^1 (\O_2 )^2 :  \quad \mathbf{z} =0 \; \mbox{ on } \; \{ x_2 =-H\} \cup \{ x_2 =0\} , \; \mathbf{z} \; \mbox{ is } L-\mbox{periodic} \} $.
\end{corollary}
\vskip6pt We remark that in $\O_2$ we have
strong $L^2$-com\-pact\-ness of the family $\{\tilde {\cal P}^\ep\}$.
From the  properties of Tartar's restriction
operator (see \cite{Ta1980} or \cite{All97}) it follows:
  \begin{lemma} \label{L1.8}  
 The sequence
$\{ {\tilde {\cal P}}^\ep 
\}$ is strongly relatively compact in  $L^2 (\O_2)$.
\end{lemma}
Following the homogenization derivation of the Darcy law from \cite{ESP}, \cite{Ta1980},  \cite{All97} or \cite{MIK00},
we consider the following auxiliary problems in $Y_F$:

$\hbox{ For } 1\leq i\leq 2 ,\, \hbox{ find } \ \{ \mathbf{w}^i ,
\pi^i \} \in H^1_{per} (Y_F)^2 \times L^2 (Y_F), \,  \int_{Y_F}  \pi^i (y ) \, dy =0, \hbox{ such that
} $
\begin{equation}\label{1.57}
   \left\{ \begin{matrix}
   \hfill - \Delta_y \mathbf{w}^i (y) + \nabla_y \pi^i (y ) = \mathbf{e}^i &
\hbox{ in } \ Y_F \cr \hfill \hbox{ div}_y \mathbf{w}^i (y) =0  & \hbox{ in
} \ Y_F \cr \hfill \mathbf{w}^i (y ) =0 & \hbox{ on } \, (\partial Y_F
\setminus
\partial Y) \cr
 \end{matrix}  \right.
\end{equation}
 Obviously, these problems always admit  unique
solutions. Let us introduce the {\bf permeability matrix} $K$ by
\begin{equation}\label{1.58}
   K_{ij} = \int_{Y_F} \nabla_y \mathbf{w}^i : \nabla_y \mathbf{w}^j \ dy =\int_{Y_F} w^i_j \ dy ,
\; 1\leq i,j \leq 2.
\end{equation}
 Then after \cite{SP80}, permeability tensor $K$ is
symmetric and positive definite. Consequently, the {\bf drag
tensor} $K^{-1}$ is also positive definite.

\begin{proof} ({\bf Proof of Theorem \ref{Limitpress}}) $\; $
Let the function ${\hat p}^0$ be the solution for the boundary value problem
\begin{gather}
    \mbox{ div } \bigg( K  (\mathbf{f} (x) - \nabla  {\hat p}^0 )\bigg) =0\; \mbox{ in } \; \O_2  \label{PresspmAu1} \\
  {\hat p}^0  = p^0 + C^{bl}_\omega  \sigma^0_{12} (x_1) \; \mbox{ on } \; \Sigma; \quad K  (\mathbf{f} (x) - \nabla  {\hat p}^0 ) |_{\{ x_2 =-H \} } \cdot \mathbf{e}^2 =0. \label{PresspmAu2}
\end{gather}

We take as test function in (\ref{UPsystem}) $\varphi (x) \psi (y)$, with $\varphi \in C^\infty_0 (\O_2)$ and $\psi \in H^1_{per} (Y_F)^2$, div$_y \psi =0$. Then after passing to the subsequence
$$ \frac{{\cal U}^\ep }{\ep^2} \to {\cal U}^{imp} (x,y), \; \nabla \frac{{\cal U}^\ep }{\ep} \to \nabla_y {\cal U}^{imp} (x,y) \quad \mbox{ and } \; {\tilde {\cal P}}^\ep 
\to {\cal P}^{imp} (x) $$
and we have
\begin{equation}\label{Vareq}
    \int_{\O_2} \int_{Y_F} \nabla_y {\cal U}^{imp} : \nabla_y \psi \varphi \ dy dx -  \int_{\O_2} \int_{Y_F} {\cal P}^{imp} (x) \psi (y) \nabla_x \varphi (x) \ dy dx = \int_{\O_2} \int_{Y_F}  (\mathbf{f} - \nabla {\hat p}^0 ) \psi (y)  \varphi (x) \ dy dx,
\end{equation}
implying
\begin{equation}\label{pseudodarcy}
    {\cal U}^{imp} (x,y) = \sum_{j=1}^2 \mathbf{w}^j (y) (f_j (x) - \frac{\p ({\hat p}^0 + {\cal P}^{imp} ) }{\p x_j}) \; \mbox{(a.e.) in } \; \O_2 .
\end{equation}
Consequently, we obtain ${\hat p}^0 + {\cal P}^{imp} \in H^1 (\O_2 )$. 


 By Corollary \ref{T4.17c} it holds  that $\ep^{-1} \nabla {\cal U}^\ep \rightharpoonup 0$ strongly in $L^2 (\O_1)^2$.
Next  taking $\varphi \in C^\infty_0 (\O)$ and using a priori estimates (\ref{4.84}), the variational equation (\ref{UPsystem}) yields a generalized form of (\ref{Vareq}) leading to
\begin{equation}\label{bdrypress}
    {\cal P}^{imp} =0 \quad \mbox{on} \quad \Sigma .
\end{equation}
Averaging div ${\cal U}^\ep$ in $\O_2$ results in
\begin{equation*}
    \mbox{ div } \bigg( K  (\mathbf{f} (x) - \nabla ( {\cal P}^{imp} + {\hat p}^0 ))\bigg) =0\; \mbox{ in } \; \O_2 .
\end{equation*}
Hence the function ${\cal P}^{imp} + {\hat p}^0 $ is $L-$ periodic in $x_1$ and satisfies
\begin{gather}
    \mbox{ div } \bigg( K  (\mathbf{f} (x) - \nabla ( { P}^{imp} + {\hat p}^0 ))\bigg) =0\; \mbox{ in } \; \O_2 ,  \label{PresspmN} \\
 {\cal P}^{imp} + {\hat p}^0  = p^0 + C^{bl}_\omega  \sigma^0_{12} (x_1) \; \mbox{ on } \; \Sigma; \quad K  (\mathbf{f} (x) - \nabla ({\cal P}^{imp} + {\hat p}^0 )) |_{\{ x_2 =-H \} } \cdot \mathbf{e}^2 =0, \label{Presspm2N}
\end{gather}
and we have ${\cal P}^{imp} =0$.  Let ${\tilde p}^0$ be the solution to the problem (\ref{Presspm})-(\ref{Presspm2}). Using Proposition \ref{proxio} we find out that  ${\tilde p}^0$ and ${\hat p}^0$ differ for $C\ep $ in any $H^k (\O_2)$, $k\in \mathbb{N}$. Hence we have established (\ref{ConcDarcy})-(\ref{ConcPression}). It remains to prove the last stated result 
i.e. 
\begin{equation}\label{pressest}
    || p^\ep - p^0 ||_{H^{-1/2} (\Sigma)} \leq C \sqrt{\ep} .
\end{equation}
\vskip0pt We use the variational equation (\ref{UPsystem}) with test function having support in $\O_1$ and  Corollary \ref{T4.17c} to obtain
\begin{equation}\label{Weaktrace1}
    || \mbox{ div } \ (\nabla {\cal U}_2^\ep - {\cal P}^\ep \mathbf{e}^2 - 2\beta^{bl, \ep}_2 \frac{d \sigma^0_{21}}{d x_1} \mathbf{e}^1 )||_{L^2 (\O_1 )} + || \nabla {\cal U}_2^\ep - {\cal P}^\ep \mathbf{e}^2 - 2\beta^{bl, \ep}_2 \frac{d \sigma^0_{21}}{d x_1} \mathbf{e}^1 ||_{L^2 (\O_1 )^2} \leq C\ep.
\end{equation}
Estimate (\ref{Weaktrace1}) implies the following estimate for the trace
\begin{equation}\label{Weaktrace2}
    || \frac{\p {\cal U}_2^\ep}{\p x_2} - {\cal P}^\ep  ||_{H^{-1/2} (\Sigma )} \leq C\ep .
\end{equation}
Next, we remark that
$$   \frac{\p {\cal U}_2^\ep}{\p x_2} = \mbox{ div } {\cal U}^\ep -  \frac{\p {\cal U}_1^\ep}{\p x_1}  
 $$
and on $\Sigma$, using Theorem \ref{T4.17}, we obtain
\begin{equation}\label{Viscstress}
  || {\cal P}^\ep  ||_{H^{-1/2} (\Sigma )} \leq    ||  \frac{\p {\cal U}_2^\ep}{\p x_2} ||_{H^{-1/2}  (\Sigma )} +C\ep \leq ||   \frac{\p {\cal U}_1^\ep}{\p x_1}||_{H^{-1/2}  (\Sigma )} +
C \ep .
\end{equation}
A direct calculation shows that
$$ || [ \frac{\p {\cal U}_2^\ep}{\p x_2} - {\cal P}^\ep  ] ||_{L^\infty (\Sigma )} \leq C\ep ,  $$
and our result is valid for the traces taken from either unconfined side or from the side corresponding to the porous medium.
\end{proof}

\section{Appendix 1: Very weak solutions to the Stokes system
in $\Omega_1$}\label{WeakStokes} \vskip5pt

 Let $\mathbf{G}_1 \in L^2 (\O_1)^2$ , ${G}_2 \in L^2 (\O_1)^4$, and $\xi\in L^2(\Sigma )^2$. We consider the following Stokes system in $\Omega_1 $:
\begin{equation}\label{5.1}
    \left\{
      \begin{array}{ll}
      \displaystyle  -\Delta \mathbf{b} + \nabla P  = \mathbf{G}_1 + \mbox{ div } G_2   \;  \hbox{ in }
         \hskip1pt  \O_1  ;&\\ \noalign{\vskip+4mm}
       \hbox{div\  } \mathbf{b} =0  \quad \hbox{ in } \hskip3pt \O_1
       , &\\ \noalign{\vskip+4mm}
       \mathbf{b} =\xi \,  \hbox{ on } \hskip3pt \Sigma_T = \Sigma \cup \{ x_2 =h \}  , &\\ \\
       \{ \mathbf{b} , P \}  , \;  \hbox{is L-periodic in } \; x_1 . &
      \end{array}
    \right.
\end{equation}

Our aim is to show the existence of a very weak solution $(\mathbf{b}, P)\in L^2(\O_1)^2\times H^{-1}(\O_1)$ to problem  (\ref{5.1}). To this end, we use  the transposition method from \cite{CO}.

Thus, let us test problem (\ref{5.1}) with a smooth test function $(\mathbf{\Phi} , \pi)$, satisfying $\mathbf{\Phi}=0$ on $\Sigma_T $ and being  $L$-periodic in $x_1 $. Furthermore, $\pi$ is $L$-periodic in $x_1 $. We obtain
\begin{eqnarray}
&& < \mathbf{G}_1 +  \mbox{ div } G_2 , \mathbf{\Phi} > =
  <- \mbox{ div } (\nabla \mathbf{b} - P I) , \mathbf{\Phi} >   =-\int_{\O_1}P\hskip 3pt \hbox{ div } \hskip3pt \mathbf{\Phi} \ dx  +\notag \\
&&   \int_{\Sigma_T}(2
D(\mathbf{\Phi})-\pi I)\nu\xi \ dS dt + \int_{\O_1} \mathbf{b} \cdot \bigg( -
 \Delta \mathbf{\Phi} + \nabla \pi \bigg) \
dx . \label{Duo}
\end{eqnarray}
Let $(\mathbf{g},s) \in W^{q-2, r} (\O_1)^2 \times W^{q-1, r} (\O_1)$,  $1<r<+\infty$, $1\leq q\leq 2$, and
$
\mathcal{H}= \{ z\in   W^{q-1, r} (\O_1), \, \int_{\Omega_1} z \ dx=0 \},
$
and denote by $ \mathcal{H}^* $ its dual. Let now $\{\mathbf{\Phi},\pi\}$ be given by
\begin{equation}\label{5.2}
\left\{
  \begin{array}{ll}
    {\displaystyle  -  \Delta \mathbf{\Phi}  + \nabla \pi
    =\mathbf{g}} \; \hbox{ in } \hskip 3pt \O_1 , &\\ \noalign{\vskip+4mm}
    \hbox{div\ }\mathbf{\Phi}=s  \quad \hbox{ in } \hskip 3pt \O_1 , &\\  \noalign{\vskip+4mm}
    \Phi=0 ,  \hbox{ on }
\hskip 3pt \Sigma_T ,  \quad
    \{ \mathbf{\Phi} , \pi \}  \quad \hbox{is L-periodic in } \; x_1
    .&
  \end{array}
\right.
\end{equation}
After the elliptic regularity for the Stokes system in \cite{Tem}, for $q\neq 1+ 1/r $ we obtain
$
\mathbf{\Phi}\in  W^{q,r}(\Omega_1)^2, \,\,
 \pi\in W^{q-1,r} (\Omega_1),
$
with $\int_{\Omega_1}\pi= 0$,
and the following estimates hold
\begin{equation}\label{EstStokes}
\Vert\mathbf{\Phi} \Vert_{W^{q,r} (\Omega_1)^2} + \Vert \nabla \pi \Vert_{W^{q-2, r} (\Omega)} dt \leq C\left( \Vert \mathbf{g} \Vert_{W^{q-2 ,r} (\O_1)^2} + \Vert \nabla s \Vert_{W^{q-2, r} (\O_1)}\right).
\end{equation}

Now, analogously to the approach in \cite{CO} where the stationary Stokes system was treated,  for $q>1 + 1/r$, we consider the linear form
\begin{equation}\label{5.3}
\ell (\mathbf{g}, s)=  \langle \mathbf{G}_1 +  \mbox{ div }G_2 ,  \mathbf{\Phi} \rangle_{\O_1}  - \langle \xi ,
(\nabla\mathbf{\Phi} -\pi I)\nu  \rangle_{\Sigma_T} ,
\end{equation}
where $(\mathbf{\Phi}, \pi )$ is given by (\ref{5.2}). Since $(\mathbf{\Phi}, \pi )$ satisfies (\ref{EstStokes}), the linear form $\ell : W^{q-2, r} (\O_1)^2 \times  \mathcal{H}\rightarrow \mathbb{R} $ is continuous, and we set 
 \begin{definition}  \label{D1} (A very weak variational formulation for the Stokes
problem (\ref{5.1})).
 $\{\mathbf{b}, P\}$ is  a very weak solution of the problem (\ref{5.1})  if
\begin{equation}
 \{ \mathbf{b}, P \}\in W^{2-q, r/(r-1)}(\O_1)^2\times \mathcal{H}^* \label{5.4}
\end{equation}
and satisfies
\begin{equation}
 \langle \mathbf{g} , \mathbf{b} \rangle_{ \Omega_1}     -\langle
P,s\rangle_{\mathcal{H}^*,\mathcal{H}}=\ell (\mathbf{g}, s), \quad
\forall \mathbf{g}\in L^r(\O_1)^2, \; \forall s\in \mathcal{H}.
\label{5.5}
\end{equation}
\end{definition}
\noindent Because of the linearity and continuity of $\ell $, Riesz's theorem implies
\begin{proposition} \label{D2} Let $1< r<+\infty$ and $1+1/r < q\leq 2$ and $< \xi_2 , 1 >_{\Sigma_T}=0$. Then, there exists a unique very
weak solution $\{\mathbf{b} , P\} $ for (\ref{5.1}). It satisfies
the following estimates
\begin{equation}\label{5.6}
\| \mathbf{b}  \|_{W^{2-q , r/(r-1)}(\O_1)^2}\le c\Big\{ \| \mathbf{G}_1
\|_{ L^{1} (\Omega_1)^2}+ \| G_2 \|_{ W^{1-q, r/(r-1)}
(\Omega_1)^4}+\|\xi\|_{W^{1+ 1/r -q, r/(r-1)}(\Sigma_T)^2} \Big\}.
\end{equation}
\end{proposition}
Another approach is to use directly the result from the article \cite{FKV}, which reads
\begin{proposition} \label{D3} Let $\mathbf{G}_1 =0$ and $G_2=0$. Then for $\xi \in L^2 (\Sigma_T )$, $\int_{\Sigma_T} \xi_2 =0$, there exists a unique very
weak solution $\{\mathbf{b} , P\} $ of (\ref{5.1}), satisfying
the following estimates
\begin{equation}\label{5.6A}
\| \mathbf{b}  \|_{H^{1/2}(\O_1)^2}\le c\|\xi\|_{L^{2}(\Sigma_T)^2} .
\end{equation}
Furthermore,
\begin{equation}\label{5.6B}
\| | x_2 |^{1/2} \nabla \mathbf{b}  \|_{L^{2}(\O_1)^2} + \| | x_2 |^{1/2} \pi  \|_{L^{2}(\O_1)^2}\le c\|\xi\|_{L^{2}(\Sigma_T)^2} .
\end{equation}
\end{proposition}

\section{Appendix 2: Navier's boundary layer and compressibility corrections}\label{Navier}

 In this Appendix, for completeness of the paper, we recall the derivation of
 Navier's boundary layer developed in \cite{JaMi2} and \cite{JM00} and presented also in \cite{JaegMik09}.

As observed in hydrology, the phenomena relevant to the boundary
occur in a thin layer surrounding the interface between a porous
medium and a free flow.  In this Appendix  we are going to present
a sketch of the  construction of the main boundary layer, used for
determining the coefficient $\alpha$ in (\ref{AUX}) and  the coefficient $C^{bl}_\omega$
in the interface pressure jump law (\ref{Presspm2}).
 Since the
law by Beavers and Joseph is an example of the Navier slip
condition, we call it {\bf Navier's boundary layer}.
\vskip3pt
In addition to the notations from subsection \ref{ExpDarcy}, we introduce
 the interface $S=(0,1)\times \{ 0\} $, the free fluid slab $Z^+ = (0,1) \times (0, +\infty )$
and the semi-infinite porous slab $Z^- =
\displaystyle\cup_{k=1}^\infty ( Y_F -\{ 0,k \} )$. The flow region
is then $Z_{BL} = Z^+ \cup S \cup Z^- $. \vskip3pt We consider the
following problem: \vskip1pt Find $\{ \beta^{bl} , \omega ^{bl} \} $
with square-integrable gradients satisfying
\begin{gather} -\triangle _y \beta^{bl} +\nabla_y \omega ^{bl} =0\qquad \hbox{ in
} Z^+ \cup Z^- \label{BJ4.2}\\ \div_y  \beta^{bl} =0\qquad \hbox{ in }
Z^+ \cup Z^- \label{4.3} \\ \bigl[ \beta^{bl} \bigr]_S (\cdot , 0)= 0
 \quad \mbox{ and } \quad \bigr[ \{ \nabla_y \beta^{bl}
-\omega^{bl} I \}  \mathbf{e}^2 \bigl]_S (\cdot , 0) = \mathbf{e}^1 \ \hbox{ on }
S\label{4.5)} \\ \beta^{bl} =0 \quad \hbox{ on }
\displaystyle\cup_{k=1}^{\infty} ( \p Y_s -\{ 0,k \} ), \qquad \{
\beta^{bl} , \omega^{bl} \} \, \hbox{ is } 1- \hbox{periodic in }
y_1 \label{4.6} \end{gather} By Lax-Milgram's lemma, there exists  a
unique $\beta^{bl} \in L^2_{loc} (Z_{BL} )^2, \; \nabla _y z \in L^2
(Z_{BL})^4$  satisfying (\ref{BJ4.2})-(\ref{4.6}) and
 $\omega^{bl} \in L^2_{loc}
(Z^+ \cup Z^- )$, which is unique up to a constant and satisfying
(\ref{BJ4.2}). {We note that due to the incompressibility and the continuity of $\beta^{bl}$ on $S$, considering $\nabla \beta^{bl}$ or the symmetrized gradient $(\nabla +\nabla^t ) \beta^{bl}$ is equivalent}.\vskip0pt The goal of this subsection is to show that  system (\ref{BJ4.2})-(\ref{4.6}) describes a
boundary layer, i.e. that $\beta^{bl} $ and $ \omega^{bl} $
stabilize exponentially towards  constants, when $\vert y_2\vert \to
\infty$. Since we are studying an incompressible flow, it
is useful to prove properties of the conserved averages.
\begin{lemma} \label{L4.1} (\cite{JaMi2}). Any solution $\{ \beta^{bl} ,
\omega^{bl} \}$ satisfies
\begin{gather}
\int^1_0 \beta^{bl}_2 (y_1 , b) \ d y_1 =0 , \quad \forall b\in
\RR 
\; \mbox{ and } \; \int^1_0 \omega^{bl}  (y_1 , b_1 ) \ d y_1 =
\int^1_0  \omega^{bl} (y_1 , b_2 ) \ d y_1 , \; \forall b_1 >
b_2 \geq 0, \label{4.8}\\ \int^1_0 \beta^{bl}_1  (y_1 , b_1) \ d y_1
=\int^1_0 \beta^{bl}_1   (y_1 , b_2) \ d y_1
= - \int_{Z_{BL}} \vert \nabla \beta^{bl} (y) \vert^2 \ dy ,  \qquad \forall b_1
> b_2 \geq 0.
\label{4.11}\end{gather}
\end{lemma}
\begin{proposition} \label{P4.3} (\cite{JaMi2}). Let
\begin{equation}\label{4.15}
    C^{ bl}_1  = \int_0^1 \beta^{ bl}_1 (y_1, 0) dy_1 .
\end{equation}
\begin{equation}\label{4.16}
  \hbox{Then, for every } \; y_2 \geq 0\; \hbox{ and } \; y_1 \in (0,1), \;  \vert \beta^{bl} (y_1 , y_2 ) - ( C^{bl}_1 , 0) \vert \leq C e^{-\delta y_2}
, \quad \forall \delta < 2\pi .
\end{equation}
\end{proposition}

\begin{corollary} \label{C4.4} (\cite{JaMi2}). Let
\begin{equation}\label{4.19}
    C^{bl}_{\omega}  =\int_0^1  \omega^{bl} (y_1 , 0)\, dy_1 .
\end{equation}
\begin{equation}\label{4.20}
 \hbox{Then, for every} \quad  y_2 \geq 0\;  \hbox{ and } \; y_1 \in (0,1), \qquad \hbox{ we have } \qquad   \mid \omega^{bl} (y_1 , y_2 ) - C^{bl}_\omega \mid \leq   e^{-2\pi y_2} .
\end{equation}
\end{corollary}
\noindent In the last step we study the decay of $\beta^{bl}$ and
$\omega^{bl}$ in the semi-infinite porous slab $Z^-$.
\begin{proposition} \label{P4.7} (see \cite{JaMi2}, pages 411-412). Let   $\beta^{bl} $ and $\omega^{bl} $
be defined by (\ref{BJ4.2})-(\ref{4.6}). Then,  there exist  positive constants
$C$ and $\gamma_0$,  such that
\begin{equation}\label{4.32}
|  \beta^{bl} (y_1 , y_2 ) | +
 | \nabla \beta^{bl} (y_1 , y_2 ) |  \leq C e^{-\gamma_0 | y_2|} , \qquad \hbox{for every } \quad (y_1 , y_2 ) \in Z^- .
\end{equation}
Furthermore,  the limit $\displaystyle  \kappa_{\infty} = \lim_{k\to -\infty } \frac{1 }{ \mid Y_F \mid }
\int_{Z_k} \omega^{bl} (y) \, dy$ exists and it holds
\begin{equation}\label{4.36}
    | \omega^{bl} (y_1 , y_2 )- \kappa_{\infty } | \leq C e^{-\gamma_0 | y_2|} , \qquad \hbox{for every } \quad (y_1 , y_2 ) \in Z^- .
\end{equation}
\end{proposition}
\begin{remark} \label{R4.9} Without loosing generality, we take $\kappa_{\infty }=0$. If the geometry of $Z^-$ is axially symmetric
with respect  to reflections around the axis $y_1= 1/2$, then
$C^{bl}_\omega =0$. For the proof, we refer to \cite{JMN01}. In
\cite{JMN01} a detailed numerical analysis of the problem
(\ref{BJ4.2})-(\ref{4.6}) is given. Through numerical experiments it is shown
that for a general geometry of $Z^-$, $C^{bl}_\omega \neq 0$.
\end{remark}
It is important to be sure  that the law by
Beavers and Joseph does not depend on the position of the interface.
We have the following  result
\begin{lemma} \label{L4.20}  Let $a<0$ and let $\beta^{a, bl} $ be the
solution of (\ref{BJ4.2})-(\ref{4.6}) with $S$ replaced by $S_a = (0,1) \times
\{ a \}$, $Z^+$ by $Z^+_a =(0,1)\times (a, +\infty )$ and $Z^-_a =
Z_{BL} \setminus ( S_a \cup Z^+_a )$. Then,  it holds
\begin{equation}\label{4.101}
   C^{a, bl}_1 = C^{bl}_1 - a .
\end{equation}
\end{lemma}
 \noindent This simple result  implies the invariance of the
obtained law on the position of the interface. It is in agreement with the law of Saffman for the slip coefficient formulated in \cite{SAF}. The law was confirmed numerically by Sahraoui and Kaviany in \cite{SahKav92}. For more discussion, we refer to the book \cite{K95}, page $74$, formulas $(2.193)-(2.195)$ and page $81$, Fig. $2.22$ and formula $(2.211)$.

The reminder  of the section is devoted to auxiliary functions
correcting the compressibility effects. We define $\mathbf{Q}^{ bl} $,  by
\begin{gather}
    \mbox{ div}_{y} \mathbf{Q}^{bl} (y) = \beta^{bl}_1 (y) -  C^{bl}_1 H(y_2)
    \quad \mbox{ in  } \; Z^+ \cup Z^-, \label{Div1} \\
         \mathbf{Q}^{ bl} = 0 \;  \hbox{on }  \; \displaystyle\cup_{k=1}^{\infty} ( \p Y_s -\{ 0,k \} ), \quad
          \mathbf{Q}^{ bl}  \; \hbox{is 1-periodic in } \; y_1   \label{Div2} \\
   [\mathbf{Q}^{ bl}]_S =  \mathbf{e}^2 \int_{Z_{BL}} ( C^{bl}_1 H(y_2) - \beta^{bl}_1 (y)) \ dy = - \mathbf{e}^2 \int_{Z^{-}} \beta^{bl}_1 (y) \ dy .
      \label{Div3}
\end{gather}
\begin{proposition}\label{Divdecay} (see \cite{JaMi2}, page 411)
Problem (\ref{Div1})-(\ref{Div3}) has at least one solution $\mathbf{Q}^{ bl} \in H^1 (Z^+ \cup Z^- )^2 \cap C^\infty_{loc}(Z^+ \cup Z^- )^2 $. Furthermore, $\mathbf{Q}^{bl} \in W^{1,q} (Z^+)^2$,   $\mathbf{Q}^{ bl} \in W^{1,q} (Z^-)^2$, for all $q\in [1, +\infty )$ and there exists $\gamma_0 >0$ such that
\begin{equation}\label{Div4}
    e^{\gamma_0 y_3 } \mathbf{Q}^{ bl} \in H^1 (Z^+ \cup Z^- )^2 .
\end{equation}
\end{proposition}

\end{document}